\documentclass[smallextended,natbib]{svjour3r}
\journalname{Annals of the Institute of Statistical Mathematics}
\smartqed  
\usepackage{graphicx}
\newcommand{\ii}{\sqrt{-1}}

\newcommand\Cov{\mathrm{Cov}}
\newcommand\Cum{\mathrm{Cum}}
\newcommand\diag{\mathrm{diag}}
\newcommand\tr{\mathrm{tr}}
\newcommand\len{\mathrm{len}}

\newcommand\C{\mathbb{C}}

\newcommand\QED{\hfill$\Box$}

\usepackage{amsmath,amssymb,amsfonts}

\begin{document}

\title{Graph presentations for moments of noncentral Wishart distributions
 and their applications 
}


\author{
Satoshi Kuriki and Yasuhide Numata
}


\institute{S.\ Kuriki \at
Institute of Statistical Mathematics, ROIS,
10-3 Midoricho, Tachikawa, Tokyo 190-8562, Japan.
\email{\texttt{kuriki@ism.ac.jp}}
\and
Y.\ Numata \at
Department of Mathematical Informatics, University of Tokyo, 
7-3-1 Hongo, Bunkyo-ku, Tokyo 113-0033, Japan;
Japan Science and Technology Agency (JST), CREST.
\email{\texttt{numata@stat.t.u-tokyo.ac.jp}}
}

\date{}

\maketitle

\begin{abstract}
We provide formulas for the moments of the real and complex noncentral
Wishart distributions
of general degrees.  The obtained formulas for the real and complex cases are
described in terms of the undirected and directed graphs, respectively.
By considering degenerate cases, we give explicit formulas
for the moments of bivariate chi-square distributions and $2\times 2$
Wishart distributions by enumerating the graphs.
Noting that the Laguerre polynomials can be considered to be moments of
a noncentral chi-square distributions formally, we demonstrate 
a combinatorial interpretation of the coefficients of the Laguerre polynomials.
\end{abstract}

\section{Introduction}
\label{section:introduction}

For $t=1,\ldots,\nu$, let $X_t=(x_{ti})_{1\le i\le p}$ be
a $p$-dimensional random column vector distributed independently
according to the normal distribution $N_p(\mu_t,\Sigma)$
with mean vector $\mu_t=(\mu_{ti})_{1\le i\le p}$
and covariance matrix $\Sigma=(\sigma_{ij})_{1\le i,j\le p}$.
We define the (real) noncentral Wishart distribution $W_p(\nu,\Sigma,\Delta)$
by the distribution of a $p\times p$ symmetric random matrix
\begin{equation}
\label{wishart-real}
 W = (w_{ij})_{1\le i,j\le p}, \quad
 w_{ij} = \sum_{t=1}^\nu x_{ti} x_{tj},
\end{equation}
where
$$ \Delta = (\delta_{ij})_{1\le i,j\le p}, \quad
 \delta_{ij} = \sum_{t=1}^\nu \mu_{ti} \mu_{tj} $$
is the mean square matrix.
The distribution of $W$ depends on $\mu_t$'s through $\Delta$
because its moment generating function is
\begin{equation}
\label{mgf-real}
 E\bigl[e^{\tr(\Theta W)}\bigr] = \det(I-2\Theta\Sigma)^{-\frac{\nu}{2}}
 e^{\tr(I-2\Theta\Sigma)^{-1}\Theta\Delta},
\end{equation}
where $\Theta$ is a $p\times p$ symmetric parameter matrix
 (\cite{Muirhead1982}).

Note that $\Delta=0$ if and only if $\mu_t=0$ for all $t$.
The Wishart distribution with $\Delta=0$ is referred to as the central
Wishart distribution $W_p(\nu,\Sigma)$.  Conventionally,
the triplet $(\nu,\Sigma,\Omega)$ with $\Omega=\Sigma^{-1}\Delta$
is used for describing the noncentral Wishart distribution
rather than $(\nu,\Sigma,\Delta)$.  The matrix $\Omega$
is called the noncentrality matrix.  In our paper, we adopt
the triplet $(\nu,\Sigma,\Delta)$ for 
simplicity in describing theorems.

For $t=1,\ldots,\nu$, let $\begin{pmatrix} X_t \\ Y_t \end{pmatrix}$
be a $2p$-dimensional random column vector distributed independently
according to the normal distribution with mean vector
$\begin{pmatrix}\xi_t \\ \eta_t \end{pmatrix}$
and covariance matrix
$\begin{pmatrix} A & -B \\ B & A \end{pmatrix}$,
where $A$ and $B$ are $p\times p$ symmetric and skew-symmetric matrices,
respectively.
The distribution of a complex-valued random vector
$Z_t=(z_{ti})_{1\le i\le p}=X_t + \ii Y_t$
is referred to as the complex normal distribution
$CN_p(\mu_t,\Sigma)$
with mean $\mu_t=(\mu_{ti})_{1\le i\le p}=\xi_t + \ii \eta_t$
and covariance matrix
$\Sigma = (\sigma_{ij})_{1\le i,j\le p} = 2(A+\ii B)$.
Actually, $\Sigma$ is a ``covariance'' in the sense of
$$ \sigma_{ij} = E[(z_{ti}-\mu_{ti})(\overline{z_{tj}-\mu_{tj}})]. $$
Here, the overline denotes the complex conjugate.
From the complex random vectors $Z_t$,
we define the complex noncentral Wishart distribution $CW_p(\nu,\Sigma,\Delta)$
as the distribution of a $p\times p$ Hermitian random matrix
\begin{equation}
\label{wishart-complex}
 W = (w_{ij})_{1\le i,j\le p},
 \quad w_{ij} = \sum_{t=1}^\nu z_{ti} \overline{z_{tj}},
\end{equation}
where
$$ \Delta = (\delta_{ij})_{1\le i,j\le p}, \quad
 \delta_{ij} = \sum_{t=1}^\nu \mu_{ti} \overline{\mu_{tj}} $$
is the mean square parameter matrix.  As in the real case,
the distribution of $W$ depends on $\mu_t$'s through $\Delta$
since its moment generating function is
\begin{equation}
\label{mgf-complex}
 E\bigl[e^{\tr(\Theta W)}\bigr] = \det(I-\Theta\Sigma)^{-\nu}
 e^{\tr(I-\Theta\Sigma)^{-1}\Theta\Delta},
\end{equation}
where $\Theta$ is a $p\times p$ Hermitian parameter matrix.
See \cite{Goodman1963} for the central case.

The primary purpose of this paper is to obtain expressions for the moments
$E[w_{ab}w_{cd}\cdots w_{ef}]$
of the real and complex noncentral Wishart distributions in terms of graphs,
 where
$a,b,c,d,\ldots,e,f\in\{1,\ldots,p\}$ are arbitrary indices.

Considering the cases where the mean vectors $\mu_t$ and the covariance
matrix $\Sigma$ take particular values, we will obtain several identities
of moments of some distributions associated with the Wishart distributions.
We shall see that the derivations are reduced to enumerating graphs of
various types.  This is the secondary purpose of our paper.

The Wishart distribution originates with a paper by \cite{Wishart1928}
around 80 years ago.
Since then, it is considered to be a fundamental distribution not only
in mathematical statistics but also in other fields such as
random matrices theory and signal processing
 (e.g., 
\cite{Bai1999}, \cite{Maiwald-Kraus2000}).
Despite this, the structure of moments of the Wishart distributions is
still an active research topic.  In the central case,
\cite{Lu-Richards2001}, \cite{Graczyk-etal2003}, and \cite{Graczyk-etal2005}
provided formulas for moments of the real and complex Wishart distributions.
These studies are based on expansions of the moment generating functions
of the central Wishart distributions
$\det(I-2\Theta\Sigma)^{-\frac{\nu}{2}}$.
The graph presentations of moments have also been discussed in these studies.
Prior to these studies, it was known that terms in the expansion of
$\det(I-Y)^{-\alpha}$ around $Y=0$ have some combinatorial structures
(e.g., \cite{Vere-Jones1988}),
which are closely related to the problem of Wishart moment.
More recently, \cite{Letac-Massam2008} provided a method to calculate moments
of the noncentral Wishart distribution by combining expansions of
the moment generating function $\det(I-2\Theta\Sigma)^{-\frac{\nu}{2}}$
and the ``noncentral part''
$e^{\tr(I-2\Theta\Sigma)^{-1}\Theta\Delta}$
 in (\ref{mgf-real}). 

The outline of this paper is as follows.
In Section \ref{section:real}, we treat the real noncentral Wishart matrices.
Formulas for the general terms of moments of the Wishart distributions
are given in terms of undirected graphs.  This is an extension of
\cite{Takemura1991}, who treated the central case.  Then,
by letting the mean vectors $\mu_t$ and the covariance matrix $\Sigma$
have particular kinds of structures, we obtain explicit formulas for moments of
the noncentral chi-square distribution,
\cite{Kibble1941}'s bivariate chi-square (gamma) distribution,
and the $2\times 2$ central Wishart distribution with $\Sigma=I$.
Noting a formal correspondence between the moments of the noncentral
chi-square distributions and the Laguerre polynomials, we will show that
the coefficients of the Laguerre polynomials have a combinatorial
interpretation.

In Section \ref{section:complex}, we treat the case of
the noncentral complex Wishart matrices.
Major parts of discussions are parallel to the real case.
One remarkable difference is that moments in the complex case are not
described in terms of undirected graphs but directed graphs.  


\section{Moments of the real noncentral Wishart distribution}
\label{section:real}

\subsection{A graph presentation}

In this subsection, we provide a graph presentation formula for moments of 
the real noncentral Wishart distributions of general degrees.  Our results
are generalizations of Theorem 4.3 of \cite{Takemura1991}
 where the central real Wishart matrices are treated.
Our basic tool is the following formula for moments of Gaussian random
vectors.  This is just a moment-cumulant relation in the Gaussian case.
For the proof, see \cite{McCullagh1987}.
In the central case $\mu=0$, this is sometimes referred to as the Wick formula.

\begin{lemma}[Moment of the real normal distribution]
\label{lemma:wick-real}
Let $X=(x_i)$ be a Gaussian random vector with mean $\mu=(\mu_i)$,
and covariance matrix $\Sigma=(\sigma_{ij})$.  Then,
$$
E[x_1 x_2\cdots x_n]
 = \sum \sigma_{i_1 i_2} \cdots \sigma_{i_{2m-1} i_{2m}}
        \mu_{i_{2m+1}} \cdots \mu_{i_n},
$$
where the summation is taken over all partitions of
$n$ indices $\{1,2,\ldots,n\}$ into unordered $m$ pairs and $n-2m$ singletons
$$ (i_1,i_2),\ldots,(i_{2m-1},i_{2m}),(i_{2m+1}),\ldots,(i_n). $$
\end{lemma}

\begin{remark}
Although Lemma \ref{lemma:wick-real} just states an expression for
$E[x_1 x_2\cdots x_n]$, it indeed gives general forms of the moments
$E[x_a x_b\cdots x_c]$ by considering a degenerate case.
For example, we have
$E[x_1 x_2^2]=E[\widetilde x_1 \widetilde x_2 \widetilde x_3]$,
where
$$
 \begin{pmatrix} \widetilde x_1 \\ \widetilde x_2 \\ \widetilde x_3
 \end{pmatrix} \sim
 N_3\left(\begin{pmatrix} \mu_1 \\ \mu_2 \\ \mu_2 \end{pmatrix},
 \begin{pmatrix}
 \sigma_{11} &  \sigma_{12} &  \sigma_{12} \\
 \sigma_{21} &  \sigma_{22} &  \sigma_{22} \\
 \sigma_{21} &  \sigma_{22} &  \sigma_{22} \end{pmatrix}\right).
$$
Throughout the paper, we will use this degeneracy argument many times.
\end{remark}

Let $X_t =(x_{ti})$ ($t=1,\ldots,\nu$) be independent Gaussian random
vectors with mean $\mu_t$ and covariance matrix $\Sigma$.
Let $W=(w_{ij})$ be a Wishart matrix made of $X_t$'s as in (\ref{wishart-real}).
In the following, we give a formula for the moment
$E[w_{ab} w_{cd}\cdots w_{ef}]$ with $a,b,c,d,\ldots,e,f$ arbitrary indices.
By applying the degeneracy argument again, we can restrict our attention to
the moment $E[w_{12} w_{34}\cdots w_{2n-1,2n}]$ without loss of generality.
For example
$E[w_{11} w_{12}^2]=E[\widetilde w_{12} \widetilde w_{34} \widetilde w_{56}]$,
where
$(\widetilde w_{ij})\sim
 W_6(\nu,(\widetilde\sigma_{ij}),(\widetilde\delta_{ij}))$,
$$ (\widetilde\sigma_{ij},\widetilde\delta_{ij}) = \begin{cases}
 (\sigma_{11},\delta_{11}), & i,j\in\{1,2,3,5\}, \\
 (\sigma_{12},\delta_{12}), & i\in\{1,2,3,5\},\,j\in\{4,6\}, \\
 (\sigma_{21},\delta_{21}), & i\in\{4,6\},\,j\in\{1,2,3,5\}, \\
 (\sigma_{22},\delta_{22}), & i,j\in\{4,6\}. \end{cases}
$$ 

Let $V = \{1,2,\ldots,2n-1,2n\}$ be the set of indices
appearing in the argument of the expectation $E[w_{12} \cdots w_{2n-1,2n}]$.
In the following, we consider an undirected graph whose vertices are
the elements of $V$.
First consider an undirected graph $G_0=(V,E_0)$ with the edges
$$ E_0 = \{(1,2),\ldots,(2n-1,2n)\}. $$
For each partition of $\{1,2,\ldots,2n-1,2n\}$
into $m$ pairs and $2n-2m$ singletons,
\begin{equation}
\label{i}
 (i_1,i_2),\ldots,(i_{2m-1},i_{2m}),(i_{2m+1}),\ldots,(i_{2n}),
\end{equation}
we define a set of undirected edges
$$ E = \{(i_1,i_2),\ldots,(i_{2m-1},i_{2m})\}. $$
By adding the edges of $E$ to $G_0$, we have a graph
\begin{equation}
\label{G}
G=(V,E_0\cup E).
\end{equation}


Each connected component of $G$ is classified as
 a ``cycle'' (a path without terminals) and
 a ``chain'' (a path with two terminals).
For the partition (\ref{i}), the number of chains is $n-m$.
The number of cycles of $G$ is denoted by $\len(G)$.
Note that $\len(G)\le m$.
Let $(j_1,j_2),\ldots,(j_{2n-2m-1},j_{2n-2m})$ be pairs of
two terminal vertices of $n-m$ chains of $G$, and let
$$ \check E = \{ (j_1,j_2),\ldots,(j_{2n-2m-1},j_{2n-2m}) \}. $$
Using these notations, the general form for the moments is given below.

\begin{theorem}[Moment of the real noncentral Wishart distribution]
\label{theorem:moment-real}
Let $(w_{ij})\sim W(\nu,(\sigma_{ij}),(\delta_{ij}))$.  Then,
\begin{equation}
\label{moment-real}
E[w_{12}\cdots w_{2n-1,2n}]
= \sum_E \nu^{\len(G)} \sigma^E \delta^{\check E},
\end{equation}
where
\begin{align*}
& \sigma^E = \prod_{(i,i')\in E}\sigma_{ii'}
 = \sigma_{i_1 i_2}\cdots\sigma_{i_{2m-1} i_{2m}}, \\
& \delta^{\check E} = \prod_{(j,j')\in \check E}\delta_{jj'}
 = \delta_{j_1 j_2} \cdots \delta_{j_{2n-2m-1} j_{2n-2m}}.
\end{align*}
The summation $\sum_E$ is taken over all partitions of
$\{1,2,\ldots,2n\}$ of the form (\ref{i}).
\end{theorem}

\begin{example}
Consider the evaluation of the moment $E[w_{12}w_{34}w_{56}]$.
Then, $V=\{1,2,3,4,5,6\}$ and $E_0=\{(1,2),(3,4),(5,6)\}$.
There are 76 partitions of $V$ into pairs and singletons.
Figure \ref{figure:G} is the graph
$G=(V,E_0\cup E)$ for $E=\{(1,6),(2,5)\}$ ($\check E=\{(3,4)\}$).
Summing up 76 possibilities, we have the following:
\begin{align*}
E[w_{12} w_{34} w_{56}]
=& \nu^3 \sigma_{12} \sigma_{34} \sigma_{56}
 + \nu^2 \sigma_{23} \sigma_{14} \sigma_{56} [6]
 + \nu \sigma_{23} \sigma_{45} \sigma_{16} [8] \\
&+ \nu^2 \sigma_{12} \sigma_{34} \delta_{56} [3]
 + \nu \sigma_{23} \sigma_{14} \delta_{56} [6]
 + \nu \sigma_{12} \sigma_{45} \delta_{36} [12] \\
& \qquad
 + \sigma_{23} \sigma_{45} \delta_{16} [24] \\
&+ \nu \sigma_{12} \delta_{34} \delta_{56} [3]
 + \sigma_{23} \delta_{14} \delta_{56} [12] \\
&+ \delta_{12} \delta_{34} \delta_{56}.
\end{align*}
Here, $[n]$ means that there are $n$ terms of similar form.

\begin{figure*}
\begin{center}
\scalebox{0.6}{\includegraphics{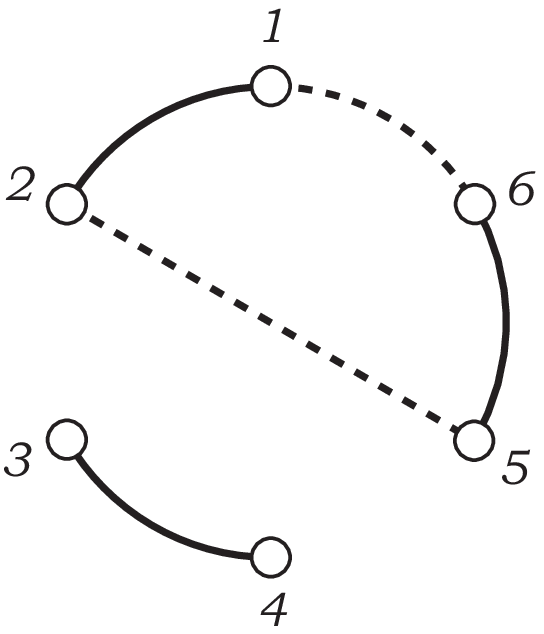}} 
\caption{Graph $G=(V,E_0\cup E)$ ($E_0$: solid line, $E$: dotted line)
presenting the term $\nu^1 \sigma_{16} \sigma_{25} \delta_{34}$
($n=6$, $m=2$, $\len(G)=1$).}
\label{figure:G}
\end{center}
\end{figure*}
\end{example}

\begin{proof}
For $i=1,\ldots,n$, let $e(i)=[(i+1)/2]$ (the integer part of $(i+1)/2$).
Noting that $w_{ij}=\sum_{t=1}^\nu x_{ti} x_{tj}$, and
from Lemma \ref{lemma:wick-real}, we have
\begin{align}
E[w_{12} & \cdots w_{2n-1,2n}] \nonumber \\
=& \sum_{t_1=1}^\nu\cdots\sum_{t_n=1}^\nu
 E[x_{t_1,1} x_{t_1,2}\cdots x_{t_n,2n-1} x_{t_n,2n}] \nonumber \\
=& \sum_{t_1}\cdots\sum_{t_n}
 E[x_{t_{e(1)},1} x_{t_{e(2)},2} \cdots x_{t_{e(2n-1)},2n-1} x_{t_{e(2n)},2n}]
 \nonumber \\
=& \sum_E \sum_{t_1}\cdots\sum_{t_n}
 \Cov(x_{t_{e(i_1)},i_1},x_{t_{e(i_2)},i_2}) \cdots
 \Cov(x_{t_{e(i_{2m-1})},i_{2m-1}},x_{t_{e(i_{2m})},i_{2m}}) \nonumber \\
& \times
 E[x_{t_{e(i_{2m+1})},i_{2m+1}}] \cdots
 E[x_{t_{e(i_{2n})},i_{2n}}].
\label{sum_E}
\end{align}
Since $\{i_1,\ldots,i_n\}=V$, the indices $i_1,\ldots,i_n$ can be
divided into connected components of the graph $G$.
Let $\{j_1,\ldots,j_{2k}\}$ be a set of vertices of a connected component.
Then, as we already pointed out, it forms either a chain
$$ (j_1,j_2),(j_2,j_3),\ldots,(j_{2k-2},j_{2k-1}),(j_{2k-1},j_{2k}), $$
or a cycle
$$ (j_1,j_2),(j_2,j_3),\ldots,(j_{2k-1},j_{2k}),(j_{2k},j_{1}), $$
and in both cases
$$ (j_1,j_2),(j_3,j_4),\ldots,(j_{2k-1},j_{2k})\in E_0. $$
Since the running indices $t_1,\ldots,t_n$ correspond to
 $n$ edges of $E_0$, and
$e(j_1)=e(j_2),e(j_3)=e(j_4),\ldots,e(j_{2k-1})=e(j_{2k})$,
the argument of the summation $\sum_E$ in (\ref{sum_E}) is
written as a product of terms of the form
\begin{align}
\sum_{t_{1}} & \cdots\sum_{t_{k}}
 E[x_{t_{1},j_1}]
 \Cov(x_{t_{1},j_2},x_{t_{2},j_3}) \cdots \nonumber \\
& \times
 \Cov(x_{t_{k-1},j_{2k-2}},x_{t_{k},j_{2k-1}})
 E[x_{t_{k},j_{2k}}]
\label{chain-real}
\end{align}
in the chain case, or
\begin{align}
\sum_{t_{1}} & \cdots\sum_{t_{k}}
 \Cov(x_{t_{1},j_2},x_{t_{2},j_3}) \cdots \nonumber \\
& \times
 \Cov(x_{t_{{k-1}},j_{2k-2}},x_{t_{{k}},j_{2k-1}})
 \Cov(x_{t_{{k}},j_{2k}},x_{t_{1},j_1})
\label{cycle-real}
\end{align}
in the cycle case.
Here, we used a reindexing
$$ t_1:=t_{e(j_1)}=t_{e(j_2)},\ldots,t_k:=t_{e(j_{2k-1})}=t_{e(j_{2k})}.
$$ 
Noting that $\Cov(x_{si},x_{tj})=1_{\{s=t\}} \sigma_{ij}$ and
$\sum_{t=1}^\nu E[x_{ti}] E[x_{tj}]=\delta_{ij}$, we see that
$(\ref{chain-real}) =
 \sigma_{j_2 j_3}\cdots\sigma_{j_{2k-2} j_{2k-1}}\delta_{j_{2k} j_1}$
and
$(\ref{cycle-real}) = \nu
 \sigma_{j_2 j_3}\cdots\sigma_{j_{2k-2},j_{2k-1}}\sigma_{j_{2k} j_1}$.
This completes the proof.
\QED
\end{proof}

\subsection{Enumeration of undirected graphs}

In this subsection, we will enumerate the graphs $G=(V,E_0\cup E)$
defined in (\ref{G})
under the condition that the number $l=\len(G)$ of cycles and
 the number $m$ of edges of $E$ are given.
Let $f_{l,m,n}$ be the number of such graphs.

Consider a degenerate noncentral Wishart matrix $W=(w_{ij})$ such that
$\sigma_{ij}\equiv 1$ and $\delta_{ij}\equiv\delta$.
This happens when every component of $X_t$ making up $W$
takes the same value with probability one.  Accordingly,
all elements of $W$ take the same value $w$, say, 
with probability one.
In this setting, (\ref{moment-real}) in Theorem \ref{theorem:moment-real}
is reduced to a moment formula for the distribution of $w$,
the noncentral chi-square distribution $\chi^2_\nu(\delta)$
with $\nu$ degrees of freedom and the noncentrality parameter $\delta$.
Using the coefficient $f_{l,m,n}$, the $n$th moment of $w$ is given as follows.
\begin{equation}
\label{Ewn}
 E[w^n] = \sum_{m=0}^n \sum_{l\ge 0} \nu^l f_{l,m,n} \delta^{n-m}.
\end{equation}

The coefficient $f_{l,m,n}$ satisfies the following recurrence formula.

\begin{lemma}
\label{lamma:recurrence-f}
\begin{equation}
\label{recurrence-f}
 f_{l,m,n} = 2(2n-m-1) f_{l,m-1,n-1} + f_{l-1,m-1,n-1} + f_{l,m,n-1},
\end{equation}
with boundary conditions
\begin{equation}
\label{boundary-f-1}
 f_{l,0,n} = \begin{cases} 1 & (l=0), \\ 0 & (l\ge 1) \end{cases}
 \quad \text{for $n\ge 1$},
\end{equation}
and
\begin{equation}
\label{boundary-f-2}
 f_{l,1,1} = \begin{cases} 0 & (l=0), \\ 1 & (l=1). \end{cases}
\end{equation}
\end{lemma} 

\begin{proof}
In the following, we sometimes
refer to an edge from $E_0$ as a ``solid line'' edge and
an edge from $E$ as a ``dashed line'' edge
as shown in Figure \ref{figure:G}.

The connected components of $G$ are classified as cycles and chains.
In each cycle, the number of solid line edges is equal to the number of
dashed line edges.  In each chain,
the number of solid line edges is one more than the number of
 dashed line edges.  
Thus, the number of connected chains is  $n-m$,
the difference between the number of solid line edges and
the number of dashed line edges.

Consider a graph $G'$ made by removing two vertices $2n-1$ and $2n$,
and all (solid line and dashed line) edges connecting to these two vertices.
Note that the (solid line) edge $(2n-1,2n)$ is deleted. 

One of the following will meet:  The edge $(2n-1,2n)$ is contained in 
(i) a cycle with 4 or more edges, or a chain with 3 or more edges
(the edges $(a,b)$, $(c,d)$, $(i,j)$, or $(k,l)$ in Figure \ref{figure:G-12});
(ii) a cycle with 2 edges
($(e,f)$ in Figure \ref{figure:G-12});
(iii) a chain consisting of a edge
($(g,h)$ in Figure \ref{figure:G-12}).

Case (i).
In the graph $G'$ made by removing two vertices $2n-1$ and $2n$, and all
connected edges, there are $n-1$ solid line edges. Since one dashed line edge
is removed together, there are $m-1$ dashed line edges.
The number of chains still remains $(n-1)-(m-1)=n-m$.
To the graph $G'$, consider adding the edge $(2n-1,2n)$ again.
There are $n-1+(n-m)=2n-m-1$ places where $(2n-1,2n)$ can be
inserted, and considering the direction of the inserted edge,
there are $2(2n-m-1)$ ways in which the insertion can be done.
By this operation, the number of dashed lines
increases by 1, whereas the number of cycles is invariant.
The contribution of the number of graphs made by this operation
to $f_{l,m,n}$ is
$$ 2(2n-m-1) f_{l,m-1,n-1}. $$

Case (ii).
Consider adding the edge $(2n-1,2n)$ to the graph $G'$ again
to make a cycle with the edge $(2n-1,2n)$ and a dashed line edge.
By this operation, both the number of dashed lines and
the number of cycles increases by 1.
The contribution of the number of graphs made by this operation
to $f_{l,m,n}$ is
$$ f_{l-1,m-1,n-1}. $$

Case (iii).
Consider adding the edge $(2n-1,2n)$ to the graph $G'$ again
to make a chain consisting of one edge $(2n-1,2n)$.
By this operation, both the number of dashed lines and
the number of cycles are invariant.
The contribution of the number of graphs made by this operation
to $f_{l,m,n}$ is
$$ f_{l,m,n-1}. $$

Summing up the three cases (i), (ii), and (iii), we obtain the recurrence
formula (\ref{recurrence-f}).
\QED
\end{proof}

\begin{theorem}
\label{theorem:generating-function}
The generating function of $f_{l,m,n}$ with respect to the number $l$ of cycles
is given by
\begin{align}
\Phi_{m,n}(\nu)
& = \sum_{l\ge 0} \nu^l f_{l,m,n} \nonumber \\
& = {n \choose m} \prod_{i=1}^m (\nu +2(n-i))
\quad (0\le m\le n,\,n\ge 1).
\label{solution-phi}
\end{align}
Here, we use a convention $\prod_{i=1}^0=1$.
\end{theorem}

\begin{figure*}
\begin{center}
\scalebox{0.6}{\includegraphics{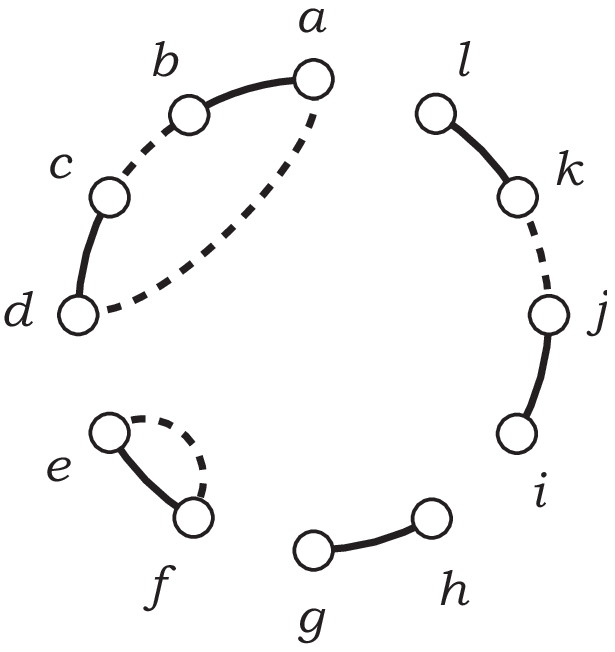}} 
\caption{A figure for the proof of Lemma \ref{lamma:recurrence-f}.}
\label{figure:G-12}
\end{center}
\end{figure*}

\begin{proof}
Noting that $f_{-1,m,n}=0$, the generating function
$\Phi_{m,n}(\nu) = \sum_{l\ge 0} \nu^l f_{l,m,n}$ has to satisfy
\begin{equation}
\label{recurrence-phi}
 \Phi_{m,n} = 2(2n-m-1) \Phi_{m-1,n-1} + \nu\Phi_{m-1,n-1} + \Phi_{m,n-1}.
\end{equation}

To solve the recurrence formula (\ref{recurrence-phi}),
 consider the boundary conditions.
From (\ref{boundary-f-1}), we have
\begin{equation}
\label{boundary1}
 \Phi_{0,n}(\nu) = 1 \quad (n\ge 1).
\end{equation}
Moreover, from (\ref{boundary-f-2}), we have
$$ \Phi_{1,1}(\nu) = \nu. $$
Furthermore, since $\Phi_{n,n-1}=0$,
\begin{align}
\Phi_{n,n}
 &= 2(n-1) \Phi_{n-1,n-1} + \nu\Phi_{n-1,n-1} 
 = (\nu + 2n-2) \Phi_{n-1,n-1} = \cdots \nonumber \\
 &= \prod_{i=1}^n (\nu +2(n-i)) \quad (n\ge 1).
\label{boundary2}
\end{align}
The recurrence formula (\ref{recurrence-phi})
 combined with the boundary conditions
 (\ref{boundary1}) and (\ref{boundary2}) determines $\Phi_{m,n}$
for all $m$ and $n$.

In the following, we see that $\Phi_{m,n}(\nu)$ in (\ref{solution-phi}) is
actually the solution for the recurrence formula (\ref{recurrence-phi}).
The boundary conditions (\ref{boundary1}) and (\ref{boundary2}) are satisfied.
We only have to make sure that (\ref{solution-phi}) really satisfies
 (\ref{recurrence-phi}).  Indeed, 
\begin{align*}
\Phi_{m,n} - \Phi_{m,n-1}
&=
 {n \choose m} \prod_{i=1}^m (\nu +2n-2i)
-{n-1 \choose m} \prod_{i=1}^m (\nu +2n-2-2i) \\
&=
 {n-1 \choose m-1}\frac{1}{m} \prod_{i=1}^{m-1} (\nu +2n-2-2i) \\
& \qquad \times
 \{ n (\nu +2n-2)-(n-m) (\nu +2n-2-2m) \} \\
&= 2(2n-m-1) \Phi_{m-1,n-1} + \nu\Phi_{m-1,n-1}.
\end{align*}
%
%
\QED
\end{proof}

\begin{corollary}
\begin{align*}
\Phi_{m,n}(1)
&= {n \choose m} (2n-1)(2n-3)\cdots (2n-2m+1) \\
&= {2n \choose 2m} (2m-1)!!
\end{align*}
is the number of undirected graphs $G$, and
\begin{align*}
\Phi_{m,n}(0)
&= {n \choose m} (2n-2)(2n-4)\cdots (2n-2m) \\
&= \frac{2^m \, n! (n-1)!}{m! \, (n-m)! (n-m-1)!}
\end{align*}
is the number of undirected graphs $G$ without cycles.
\end{corollary}

\begin{remark}
\label{remark:f-stirling}
Nonnegative integers $s_n(m,l)$ defined by a generating function
\begin{equation}
\label{mgf-stirling}
 \sum_{l=0}^m \nu^l s_n(m,l) = \prod_{i=1}^m (\nu+n-i)
\end{equation}
are called the noncentral Stirling numbers of the first kind
 (\cite{Koutras1982}).  Since
$$
\sum_{\nu\ge 0} \nu^l f_{l,m,n}
  = {n \choose m} 2^m \prod_{i=1}^m (\nu/2 +n-i) \\
  = {n \choose m} 2^m \sum_{l\ge 0} (\nu/2)^l s_n(m,l),
$$
we have
$$ f_{l,m,n} = {n \choose m} 2^{m-l} s_n(m,l). $$
\end{remark}

\subsection{Moments of the noncentral chi-square distribution and the Laguerre polynomial}

As stated in the beginning of the previous subsection,
the moment of the noncentral chi-square distribution is described
with the coefficient $f_{l,m,n}$.

In view of (\ref{Ewn}) and 
Theorem \ref{theorem:generating-function},
the $n$th moment of $w\sim\chi^2_\nu(\delta)$, 
the noncentral chi-square distribution
with $\nu$ degrees of freedom and the noncentrality parameter $\delta$,
is written as
\begin{align}
E[w^n]
&= \sum_{m=0}^n \sum_{l\ge 0} \nu^l f_{l,m,n} \delta^{n-m} \nonumber \\
&= \sum_{m=0}^n \Phi_{m,n}(\nu) \delta^{n-m} \nonumber \\
&= \sum_{m=0}^n {n \choose m} \prod_{i=1}^m (\nu +2(n-i)) \delta^{n-m}.
\label{moment-noncentralchisq}
\end{align}
This is a well-known expression
for the moment of noncentral chi-square distribution
 (e.g., \cite{Johnson-etal1995}).

\begin{remark}
\cite{Koutras1982} pointed out that moments of some noncentral distributions
are described with the noncentral Stirling numbers of the first kind.
\end{remark}

The moment generating function of the noncentral chi-square distribution
$\chi^2_\nu(\delta)$ is
$$ (1-2t)^{-\nu/2} e^{\delta t (1-2t)^{-1}}. $$
This can be obtained by
letting $\Theta=t$, $\Sigma=1$, $\Delta=\delta$ with (\ref{mgf-real}).

The Laguerre polynomials, the orthogonal polynomial systems on $(0,\infty)$
with respect to the gamma weight functions, are defined as
$$ L^{(\nu)}_n(x)
 = 2^n \frac{d^n}{d\sigma^n} f^{(\nu)}(x;\sigma) \Big|_{\sigma=1}
 \Big/ f^{(\nu)}(x;1) \quad (\nu>0), $$
where
$$ f^{(\nu)}(x;\sigma) = x^{\nu/2-1}\sigma^{-\nu/2} e^{-x/(2\sigma)} $$
 (e.g., \cite{Morris82}).  From this definition, we immediately get
the generating function of the Laguerre polynomial as
$$
 \sum_{n=0}^\infty (-1)^n \frac{t^n}{n!} L^{(\nu)}_n(x)
 = (1-2t)^{-\nu/2} e^{-(x/2)((1-2t)^{-1}-1)},
$$
which has formal coincidence with the moment generating function of
the chi-square distribution with $\nu$ degrees of freedom and
the noncentrality parameter $-x$.  Therefore, we have an expression
for the Laguerre polynomials
\begin{align*}
L^{(\nu)}_n(x)
& = (-1)^n \sum_{m=0}^n \sum_{\nu\ge 0} \nu^l f_{l,m,n} (-x)^{n-m} \\
& = (-1)^n \sum_{m=0}^n
 {n \choose m} \prod_{i=1}^m (\nu +2(n-i)) (-x)^{n-m}.
\end{align*}
This gives a combinatorial interpretation for the coefficients of the
Laguerre polynomials.

This type of combinatorial interpretation
for Hermite polynomials is widely known.
Applying a degenerate multivariate distribution
$N((\mu_i),(\sigma_{ij}))$ with $\mu_i\equiv\mu$, $\sigma_{ij}\equiv\sigma^2$
to Lemma \ref{lemma:wick-real}, we see that
the $n$th moment of the normal distribution $X\sim N(\mu,\sigma^2)$ is
$$ E[X^n] = \sum_{m=0}^{[n/2]} a_{m,n}\sigma^{2m} \mu^{n-2m}, $$
where
$$ a_{m,n} = {n \choose 2m} (2m-1)!! = \frac{n!}{(n-2m)! \, 2^m \, m!} $$
is the number of partitions of an $n$-member set into
(unordered) $m$ pairs and $n-2m$ singletons.
The Hermite polynomials are defined by
$$ H_n(x) = (-1)^n \frac{d^n}{d x^n} e^{-x^2/2} \Big/ e^{-x^2/2}. $$
The generating function of $H_n(x)$ is
$$ \sum_{n=0}^\infty \frac{t^n}{n!} H_n(x) = e^{-(x-t)^2/2} \Big/ e^{-x^2/2}
 = e^{xt -t^2/2}, $$
which coincides with the moment generating function
$e^{\mu t+\sigma^2 t^2/2}$ of the normal distribution $N(\mu,\sigma^2)$
with $\mu$ and $\sigma^2$ replaced by $x$ and $-1$, respectively.
Therefore, the $n$th Hermite polynomial is the $n$ moment of the
normal distribution $N(x,-1)$ formally 
 (\cite{McCullagh1987}, \cite{Kuriki-Takemura1996}, 
\cite{Withers-Nadarajah2006}),
and hence
$$ H_n(x) = 
 \sum_{m=0}^{[n/2]} a_{m,n} \, (-1)^m x^{n-2m}. $$

\subsection{Moments of bivariate chi-square distribution}

In this subsection, we give an explicit expression for
the moment of bivariate chi-square distributions
as the second application of the graph presentations for the moments
of the noncentral Wishart distribution.

There are several proposals for defining bivariate chi-square
 (gamma) distributions.
The distribution we will discuss here is that given by \cite{Kibble1941}.
Kibble's bivariate chi-square distribution is defined
as the distribution of the diagonal elements $(w_{11},w_{22})$ of
a $2\times 2$ central Wishart distribution $(w_{ij})\sim W_2(\nu,\Sigma)$
with
$$ \Sigma = \begin{pmatrix} 1 & \rho \\ \rho & 1 \end{pmatrix}. $$
Substituting
$\Theta=\diag(t_{11},t_{22})$ and $\Delta=0$ into (\ref{mgf-real}), 
we have the moment generating function
\begin{align}
E[e^{t_{11} w_{11}+t_{22} w_{22}}]
&= \det\begin{pmatrix} 1 - 2 t_{11} & - 2\rho t_{11} \nonumber \\
                     - 2\rho t_{22} &  1 - 2 t_{22} \end{pmatrix}^{-\nu/2}
\nonumber \\
& = (1 - 2 t_{11} - 2 t_{22} + 4 t_{11} t_{22}
 - 4 \rho^2 t_{11} t_{22})^{-\nu/2}.
\label{mgf-bivariategamma}
\end{align}

The next theorem gives an expression for the moments of general degrees
$E[w_{11}^b w_{22}^c]$.  We evaluate this as the moment
$$ E[\widetilde w_{12}\cdots\widetilde w_{2b-1,2b}
 \widetilde w_{2b+1,2b+2}\widetilde w_{2(b+c)-1,2(b+c)}], $$
where $(\widetilde w_{ij})\sim W_{2(b+c)}(\nu,(\sigma_{ij}))$ with
$$ \sigma_{ij} = \begin{cases}
 1 & (i,j\le 2b\text{ or }i,j\ge 2b+1), \\
 \rho & (\text{otherwise}). \end{cases} $$

\begin{theorem}[Moment of the bivariate chi-square distribution]
\label{theorem:bigamma}
Let $b$ and $c$ be nonnegative integers.  Then,
\begin{align*}
E[w_{11}^b w_{22}^c]
= \sum_{a=0}^{\min(b,c)} \rho^{2a} & \frac{2^a\,b!\,c!}{(b-a)!\,(c-a)!\,a!}
  \prod_{i=1}^{a}(\nu+2(a-i)) \\
& \times
 \prod_{i=1}^{b-a}(\nu+2(b-i))
 \prod_{i=1}^{c-a}(\nu+2(c-i)).
\end{align*}
\end{theorem}

\begin{remark}
\cite{Nadarajah-Kotz2006} derived an expression for $E[w_{11}^b w_{22}^c]$
including the Jacobi polynomials.  Their derivation is a use of a series of
identities of special functions, which are totally different from our
combinatorial proof given below.
\end{remark}

\begin{proof}
Assume that $b\le c$ without loss of generality.
Let $G_0=(V_1\cup V_1,E_1\cup E_2)$ be a union of two undirected graphs
$(V_i,E_i)$ ($i=1,2$), where
$$ V_1=\{ 1,2,\ldots,2b-1,2b \}, \quad E_1=\{ (1,2),\ldots,(2b-1,2b) \} $$
and
$$ V_2=\{ 2b+1,2b+2,\ldots,2(b+c)-1,2(b+c) \}, $$
$$ E_2=\{ (2b+1,2b+2),\ldots,(2(b+c)-1,2(b+c)) \}. $$

By forming $b+c$ pairs from $2(b+c)$ vertices of $V=V_1\cup V_2$,
we add $b+c$ edges joining two vertices of each pair to the graph $G_0$
to make $G$.
Let $h_{l,a,b,c}$ be the number of resulting graphs $G$ such that
the number of cycles is $l$ and the number of edges joining an
element of $V_1$ and an element of $V_2$ 
(the number of pairs consisting of
 an element of $V_1$ and an element of $V_2$) is $2a$.
Then, the moment that we want to evaluate is represented as
$$ \sum_{l,a} \nu^l \rho^{2a} h_{l,a,b,c}. $$
We divide the process of adding $b+c$ edges into three steps below
 (Figure \ref{figure:bigamma}).

Step (i).
Choose $2c-2a$ from the $2c$ vertices of $V_2$, and form $c-a$ pairs
from the $2c-2a$ vertices. 
Add $c-a$ edges joining two vertices of each pair to the graph $G_0$. 
The number of resulting graphs having $l'$ cycles is
$$ f_{l',c-a,c}. $$
Note that $2a$ vertices not chosen are the terminal vertices of $a$ chains.

Step (ii).
Form $b$ pairs from the $2b$ vertices of $V_1$, and add $b$ edges
joining two vertices of each pair.
The number of resulting graphs having $l''$ cycles is
$$ f_{l'',b,b}. $$

Step (iii).
Choose $a$ edges from the $b$ edges added in step (ii), and
make ``cuts''
at the middle of each edge, and have the $a$ chains generated in step (i)
fit in at the $a$ cut points.
Note that this operation does not alter the number of cycles.
Since the $a$ chains have directions, the number of ways
to perform this operation is
$$ 2^a \times \frac{b!}{(b-a)!}. $$

Summing up (i), (ii), and (iii), we get
$$ h_{l,a,b,c}
 = \frac{2^a\,b!}{(b-a)!} \sum_{l'+l''=l} f_{l',c-a,c} f_{l'',b,b}. $$
Therefore,
\begin{align*}
\sum_{l,a} \nu^l \rho^{2a} h_{l,a,b,c}
& = \sum_a \rho^{2a} \frac{2^a\,b!}{(b-a)!}
 \sum_{l'}\nu^{l'} f_{l',c-a,c} \sum_{l''} \nu^{l''} f_{l'',b,b} \\
& = \sum_a \rho^{2a} \frac{2^a\,b!}{(b-a)!}
 {c \choose c-a} \prod_{i=1}^{c-a} (\nu+2(c-i)) \prod_{i=1}^b (\nu+2(b-i)).
\end{align*}
\QED
\end{proof}

\begin{figure*}
\begin{center}
\scalebox{0.6}{\includegraphics{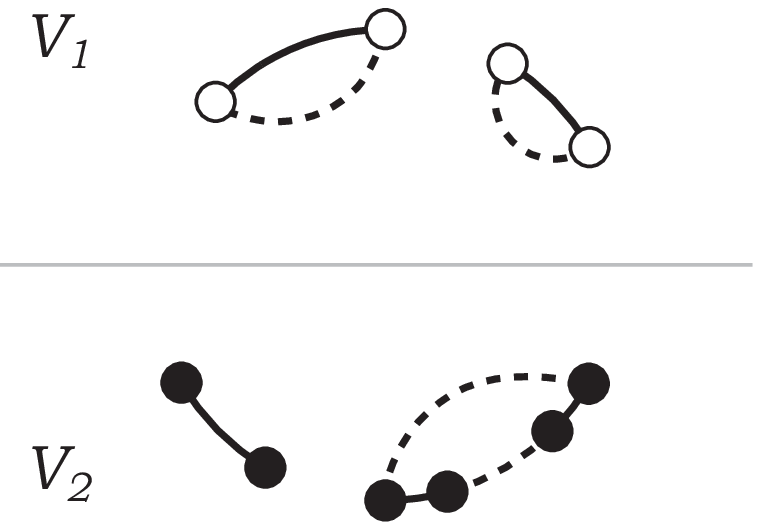}} 
\qquad
\scalebox{0.6}{\includegraphics{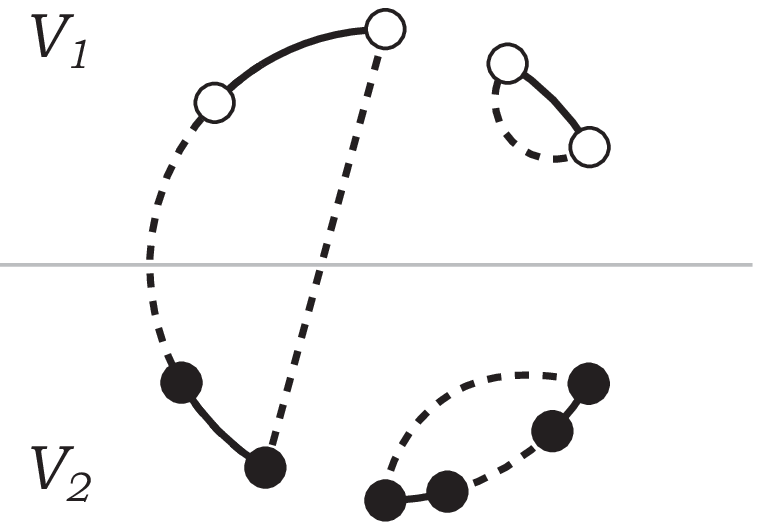}} 
\caption{
Figures for the proof of Theorem \ref{theorem:bigamma}
(Left: Steps (i), (ii), Right: Step (iii);
{\LARGE $\circ$} : Vertex of $V_1$, {\LARGE $\bullet$} : Vertex of $V_2$).
}
\label{figure:bigamma}
\end{center}
\end{figure*}

\subsection{Moments of a $2\times 2$ real Wishart distribution}

As the third example, we give an explicit expression for the moments
of a $2\times 2$ central Wishart distribution with the parameter $\Sigma=I$.
Let $(w_{ij})\sim W_2(\nu,I)$.
Substituting
$\Theta = \begin{pmatrix} t_{11} & t_{12}/2 \\ t_{12}/2 & t_{22} \end{pmatrix}$
and $\Delta=0$ into (\ref{mgf-real}), 
we have the moment generating function
\begin{align}
E[e^{t_{12} w_{12}+t_{11} w_{11}+t_{22} w_{22}}]
&= \det\begin{pmatrix} 1 - 2 t_{11} &      - t_{12} \\
                           - t_{12} &  1 - 2 t_{22} \end{pmatrix}^{-\nu/2} 
\nonumber \\
&= (1 - 2 t_{11} - 2 t_{22} + 4 t_{11} t_{22} - t_{12}^2)^{-\nu/2}.
\label{mgf-2x2-real}
\end{align}

We first show that
$$ E[w_{12}^a w_{11}^b w_{22}^c]=0 \quad \text{for $a$ odd}. $$
Let $X_t$ be Gaussian random vectors
making up the Wishart matrix.  Since $E[X_t]=0$, the distribution of
$X_t$ is invariant under the change of the sign of the first coordinate.
On the other hand, this change causes
$(w_{12},w_{11},w_{22})\mapsto (-w_{12},w_{11},w_{22})$.
This implies that
\begin{align*}
E[w_{12}^a w_{11}^b w_{22}^c]
& = E[(-w_{12})^a w_{11}^b w_{22}^c] \\
& = (-1)^a E[w_{12}^a w_{11}^b w_{22}^c].
\end{align*}
Unless $a$ is even, the left- and right-hand sides become 0.

In the following, we will derive the moment
$E[w_{12}^{2a} w_{11}^b w_{22}^c]$,
where $a,b$, and $c$ are nonnegative integers.
Although some methods to calculate this quantity are already known
 (e.g., Remark \ref{remark:geometric}), 
we demonstrate that our combinatoric approach does get the same results.

Let 
$$ V_1^0 = \{1,3,\ldots,4a-1\},\quad
 V_1^+ = \{4a+1,4a+2,\ldots,4a+2b\}, $$
and 
$$ V_2^0 = \{2,4,\ldots,4a\}, \quad
 V_2^+ = \{4a+2b+1,4a+2b+2,\ldots,4a+2b+2c\}. $$
Then, 
$$ E[w_{12}^{2a} w_{11}^b w_{22}^c]=
 E[\widetilde w_{12}\widetilde w_{34}\cdots
\widetilde w_{4a+2b+2c-1,4a+2b+2c}], $$
where $(\widetilde w_{ij})\sim W_{4a+2b+2c}(\nu,(\sigma_{ij}))$ with
$$ \sigma_{ij} = \begin{cases}
 1 & (i,j\in V_1^+,\text{ or }i,j\in V_2^+,\text{ or }
   i\in V_1^0, j\in V_2^0,\text{ or }i\in V_2^0, j\in V_1^0), \\
 0 & (\text{otherwise}). \end{cases} $$

\begin{theorem}[Moment of the $2\times 2$ real Wishart distribution]
\label{theorem:2x2wishart}
Let $a,b$, and $c$ be nonnegative integers.  Then,
\begin{equation*}
E[w_{12}^{2a} w_{11}^b w_{22}^c]
= (2a-1)!!
 \prod_{i=1}^{a}(\nu+2(a-i))
 \prod_{i=1}^{b}(\nu+2(a+b-i))
 \prod_{i=1}^{c}(\nu+2(a+c-i)).
\end{equation*}
\end{theorem}

\begin{proof}
Let $V_1=V_1^0\cup V_1^+$ and $V_2=V_2^0\cup V_2^+$.
First, define an undirected graph $G_0=(V,E_0)$ with $4a+2b+2c$ vertices 
$V=V_1\cup V_2$ and $2a+b+c$ edges
$$ E_0 = \{(1,2),(3,4),\ldots,(4a+2b+2c-1,4a+2b+2c) \}. $$
Then, consider the addition of additional $2a+b+c$ edges to the graph $G_0$ 
to make $G$ such that no edges joining $V_1$ and $V_2$ are added.
Let $h_{l,a,b,c}$ be the number of resulting graphs $G$ such that
$G$ has $l$ cycles.  Then,
$$ E[w_{12}^{2a} w_{11}^b w_{22}^c] = \sum_{l \ge 0} \nu^l h_{l,a,b,c}. $$

We divide the process of adding $2a+b+c$ edges to $G_0$ into three
steps (i), (ii), and (iii) below
 (Figure \ref{figure:2x2wishart}).

Step (i).
Let $0\le a'\le \min(a,c)$.
Choose $2c-2a'$ elements from the $2c$ vertices of $V_2^+$,
 and form $c-a'$ pairs from them.
Add $c-a'$ edges defined by the $c-a'$ pairs to the graph $G_0$.
According to this operation, $a'$ chains are newly generated.
The number of graphs with $l'$ cycles is
$$ f_{l',c-a',c}. $$

Step (ii).
Form $a$ pairs from the $2a$ vertices of $V_2^0$. 
The number of ways in which this pairing can be done is
$$ (2a-1)!! = \frac{(2a)!}{2^{a}a!}. $$
Choose $a'$ pairs from the $a$ pairs, and assign each pair to
each of the $a'$ chains generated in step (i).
Connect a vertex of the pair to one terminal vertex of the chain using a new
edge, and connect the other vertex of the pair to the other terminal vertex
of the chain using another (new) edge. (Add $2a'$ edges in total.)
The number of the 
correspondences is
$$ 2a (2a-2) \cdots (2a-2a') = \frac{2^{a'}a!}{(a-a')!}. $$
For the remaining $a-a'$ pairs, connect two vertices of each pair using
 a new edge. (Add $a-a'$ edges in total.)

The number of edges added in steps (i) and (ii) is
$(c-a')+2a'+(a-a')=a+c$.  The number of cycles is $l'$.

In steps (i) and (ii), summing the number of ways
for $0\le a'\le \min(a,c)$ yields
$$ e_{l',c,a}
 = \sum_{a'=0}^{\min(a,c)} f_{l',c-a',c} \times (2a-1)!!  
 \times \frac{2^{a'}a!}{(a-a')!}. $$
The coefficient $e_{l',c,a}$ can be combinatorially interpreted as follows:
Let $G_2=(V_2,E_2)$ with 
$E_2 = \{(4a+2b+1,4a+2b+2),\ldots,(4a+2b+2c-1,4a+2b+2c) \}$ be an undirected
graph, and add $a+c$ edges by forming $a+c$ pairs from the $2a+2c$ vertices $V_2$.
Then, $a$ chains are newly generated. 
$e_{l',c,a}$ is the number of resulting graphs having $l'$ cycles,
and the terminal vertices of the $a$ chains are elements of $V_2^0$.

Step (iii).
For all $2a+2b$ vertices of $V_1$, form $a+b$ pairs and connect
the vertices of each pair with a new edge.  (Add $a+b$ edges in total.)
According to step (ii), the $2a$ vertices of $V_1$ have already been divided 
into $a$ pairs, and the vertices of each pair have been connected with an edge.
In step (iii), the number of graphs adding new $l''$ cycles is
$f_{l'',a+b,a+b}$.

Summarizing (i), (ii), and (iii), we get
$$ h_{l,a,b,c} = \sum_{l'+l''=l} e_{l',c,a} f_{l'',a+b,a+b}, $$
and hence,
\begin{equation}
\label{convolution}
\sum_{l\ge 0} \nu^l h_{l,a,b,c}
 = \sum_{l'\ge 0} \nu^{l'} e_{l',c,a}
 \sum_{l''\ge 0} \nu^{l''} f_{l'',a+b,a+b}.
\end{equation}

For a nonnegative integer $n$, write
$$ {a \choose n} = \frac{a(a-1)\cdots (a-n+1)}{n!}. $$
Then, the generating function of the coefficient
$e_{l',c,a}$ with respect to the number of cycles $l'$ is
\begin{align*}
\sum_{l'\ge 0} \nu^{l'} e_{l',c,a}
&= (2a-1)!! \sum_{a'=0}^{\min(a,c)} \sum_{l'\ge 0} \nu^{l'} f_{l',c-a',c}
 \frac{2^{a'}a!}{(a-a')!} \\
&= (2a-1)!! \sum_{a'=0}^{\min(a,c)}
 {c \choose c-a'} \prod_{i=1}^{c-a'} (\nu+2(c-i))  \frac{2^{a'}a!}{(a-a')!} \\
&= (2a-1)!! \sum_{a'=0}^{\min(a,c)}
 {c \choose c-a'} 2^{c-a'} {\nu/2+c-1 \choose c-a'} (c-a')!
 \frac{2^{a'}a!}{(a-a')!} \\
&= (2a-1)!! \, 2^c c! \sum_{a'=0}^{\min(a,c)}
 {a \choose a'} {\nu/2+c-1 \choose c-a'} \\
&= (2a-1)!! \, 2^c c! {\nu/2+a+c-1 \choose c} \\
&= (2a-1)!! \prod_{i=1}^{c} (\nu+2(a+c-i)). 
\end{align*}
The fifth equality above is known as a convolution identity
 for two binomial coefficients.
Substituting this into (\ref{convolution}) completes the proof.
\QED
\end{proof}

\begin{figure*}
\begin{center}
\scalebox{0.6}{\includegraphics{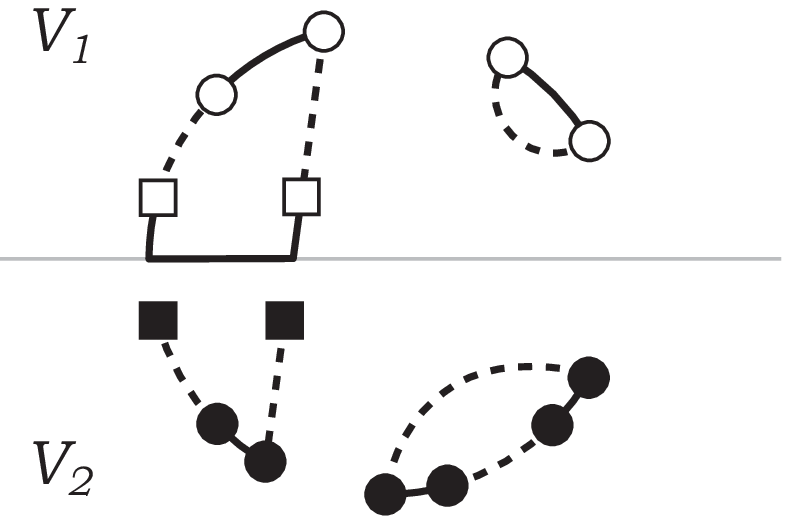}} 
\qquad
\scalebox{0.6}{\includegraphics{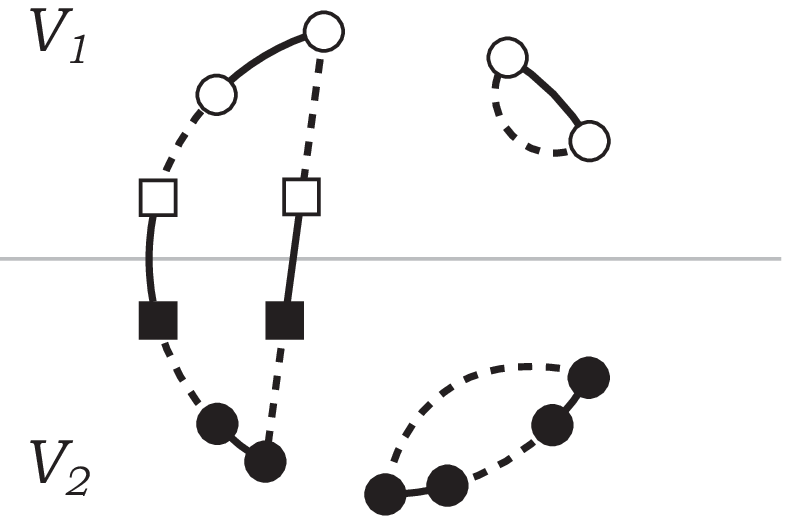}} 
\caption{
Figures for the proof of Theorem \ref{theorem:2x2wishart}
(Left: Steps (i), (ii), Right: Step (iii); 
 {\scriptsize $\Box$} : Vertex of $V_1^0$,
 {\LARGE $\circ$} : Vertex of $V_1^+$,
 {\LARGE $\hbox{\vrule width 5pt height 5pt depth 0pt}$} : Vertex of $V_2^0$,
 {\LARGE $\bullet$} : Vertex of $V_2^+$).}
\label{figure:2x2wishart}
\end{center}
\end{figure*}

\begin{remark}
\label{remark:geometric}
Theorem \ref{theorem:2x2wishart} can also be shown by the following geometric
consideration.  Let $X_t=\begin{pmatrix} x_{t1} \\ x_{t2} \end{pmatrix}$ be
Gaussian random vectors making up the $2\times 2$ Wishart matrix
 $(w_{ij})\sim W_2(\nu,I)$.
Write $X_{(i)}=(x_{1i},\ldots,x_{\nu i})'$ ($i=1,2$),
where ${}'$ denotes the matrix transposition.
Since the $\nu$-dimensional distribution of $X_{(i)}$ is invariant
under the orthogonal transformation preserving the norm
$\Vert X_{(i)}\Vert=\sqrt{X_{(i)}'X_{(i)}}$, four quantities
$\Vert X_{(i)}\Vert$, $X_{(i)}/\Vert X_{(i)}\Vert$ for $i=1,2$ are
independently distributed.  Thus,
\begin{align*}
E[w_{12}^{2a} w_{11}^b w_{22}^c]
&=E\bigl[(X_{(1)}'X_{(2)})^{2a}
 \Vert X_{(1)}\Vert^{2b}
 \Vert X_{(2)}\Vert^{2c}\bigr] \\
&=E\bigl[r^{2a}\bigr]
 E\bigl[\Vert X_{(1)}\Vert^{2(a+b)}\bigr]
 E\bigl[\Vert X_{(2)}\Vert^{2(a+c)}\bigr],
\end{align*}
where $r=X_{(1)}'X_{(2)}/(\Vert X_{(1)}\Vert\Vert X_{(2)}\Vert)$.
This calculation can be completed by noting that
$\Vert X_{(i)}\Vert^2\sim\chi^2_\nu$ and
$r^2\sim B(\frac{1}{2},\frac{\nu-1}{2})$, the beta distribution.
\end{remark}

\section{Moments of the noncentral complex Wishart distribution}
\label{section:complex}

\subsection{Preliminaries on the complex normal distribution}

In this section, we will deal with the complex noncentral Wishart matrices.
Major parts of the discussion are parallel to the real case.
One remarkable difference is that the moments in the complex case are described
in terms of directed graphs,
whereas those in the real cases are described in terms of undirected graphs.  

We begin by summarizing some preliminaries on the complex normal
distribution and the complex Wishart distribution.
Let $\overline{Z}=(\overline{z_i})$ be a complex conjugate of $Z=(z_i)$.
The following lemma is a complex version of Lemma \ref{lemma:wick-real}.

\begin{lemma}[Moment of the complex normal distribution]
\label{lemma:wick-complex}
Let $Z=(z_i)\sim CN(\mu,\Sigma)$, and let $\overline{i}=n+i$, $i=1,\ldots,n'$.
Then,
$$
 E[z_1\cdots z_n
 \overline{z_{\overline{1}}}\cdots\overline{z_{\overline{n'}}}]
 = \sum
   \sigma_{i_1 j_1}\cdots\sigma_{i_m j_m}
   \mu_{i_{m+1}}\cdots\mu_{i_{n}}
   \overline{\mu_{j_{m+1}}}\cdots\overline{\mu_{j_{n'}}},
$$
where the summation is over possible pairing
$\{(i_1,j_1),(i_2,j_2),\ldots (i_m,j_m)\}$ such that
$$ \{i_1,\ldots,i_m\} \subset \{1,2,\ldots,n\} \ \ \text{and}\ \ %
 \{j_1,\ldots,j_m\} \subset \{\overline{1},\ldots,\overline{n'}\} $$
(i.e., matching).
The other indices are
\begin{align*}
\{i_{m+1},\ldots,i_{n}\} &= \{1,\ldots,n\}\setminus\{i_1,\ldots,i_k\}, \\
\{j_{m+1},\ldots,j_{n'}\} &= \{\overline{1},\ldots,\overline{n'}\}
 \setminus\{j_1,\ldots,j_k\}.
\end{align*}
\end{lemma}

\begin{proof}
Write $Z=X+\ii Y$, $\mu=\xi+\ii\eta$, $\Sigma=2(A+\ii B)$.
Let $\theta=(\theta_i)$, $\varphi=(\varphi_i)$ be parameter column vectors.
Because of
$$ \begin{pmatrix} X \\ Y \end{pmatrix} \sim
 N\left(\begin{pmatrix} \xi \\ \eta \end{pmatrix},
        \begin{pmatrix} A & -B \\ B & A \end{pmatrix}\right), $$
and
$$ \begin{pmatrix} Z \\ \overline{Z} \end{pmatrix} = J 
 \begin{pmatrix} X \\ Y \end{pmatrix}, \qquad
J = \begin{pmatrix} I & \ii I \\ I & -\ii I \end{pmatrix}, $$
the moment generating function of $(Z,\overline{Z})$ is obtained as
\begin{align*}
E[\exp\{\theta'Z+\varphi'\overline{Z}\}]
&= \exp\biggl\{
 \begin{pmatrix} \theta' & \varphi' \end{pmatrix} J \begin{pmatrix} \xi \\ \eta \end{pmatrix} 
 + \frac{1}{2}
 \begin{pmatrix} \theta' & \varphi' \end{pmatrix} J
 \begin{pmatrix}A & -B \\ B & A\end{pmatrix}
 J' \begin{pmatrix} \theta \\ \varphi \end{pmatrix}
 \biggr\} \\
&= \exp\biggl\{
 \begin{pmatrix} \theta' & \varphi' \end{pmatrix} \begin{pmatrix} \mu \\ \overline{\mu} \end{pmatrix}
 + \frac{1}{2}
 \begin{pmatrix} \theta' & \varphi' \end{pmatrix}
 \begin{pmatrix}0 & \Sigma \\ \overline\Sigma & 0\end{pmatrix}
 \begin{pmatrix} \theta \\ \varphi \end{pmatrix}
 \biggr\} \\
&= \exp\Bigl\{ \theta'\mu + \varphi'\overline{\mu}
 + \frac{1}{2}(\theta'\Sigma\varphi+\varphi'\overline\Sigma\theta) \Bigr\} \\
&= \exp\{ \theta'\mu + \varphi'\overline{\mu} + \theta'\Sigma\varphi \}.
\end{align*}
From this, we have the joint cumulants of $(Z,\overline{Z})$ as
$$
\Cum(z_1,\ldots,z_n,\overline{z_{\overline{1}}},\ldots,\overline{z_{\overline{n'}}})
 = \begin{cases}
 \mu_1                        & (n=1,\,n'=0), \\
 \overline{\mu_{\overline{1}}} & (n=0,\,n'=1), \\
 \sigma_{1\overline{1}}        & (n=1,\,n'=1), \\
 0                            & (\text{otherwise}). \end{cases}
$$
Lemma \ref{lemma:wick-complex} is the moment-cumulant relation for
this particular cumulants.
\QED
\end{proof}

\subsection{A graph presentation}

Let $Z_t =(z_{ti})$ ($t=1,\ldots,\nu$) be independent complex Gaussian random
vectors with mean $\mu_t$ and covariance matrix $\Sigma$.
A complex Wishart matrix $W=(w_{ij})$ is constructed from $Z_t$
as given in (\ref{wishart-complex}).
In this subsection, we give a formula for the moment
$E[w_{ab} w_{cd}\cdots w_{ef}]$ with $a,b,c,d,\ldots,e,f$ arbitrary indices.
By considering the degenerate case again, without loss of generality,
 we only have to treat the moment
$E[w_{1\overline{1}} w_{2\overline{2}}\cdots w_{n\overline{n}}]$
with $\overline{i}=n+i$, $i=1,\ldots,n$.

Let $V = \{1,2,\ldots,n\}$ be the set of indices
appearing in the expectation that we want to evaluate.
In the following, we consider a directed graph whose vertices are
the elements of $V$.
Choose a subset $V_1$ of $V$ such that its cardinality is
$|V_1|=m$, and consider an injection $\pi : V_1 \to V$.
The map $\pi$ defines a set of directed pairs
$$ E = \{ (i,\pi(i)) \mid i \in V_1 \}. $$
We regard the pair $(V,E)$ as a directed graph $G$,
where $V$ is the set of vertices, and $E$ is the set of directed edges.
Note that $E$ and the pair $(V_1,\pi)$ have one-to-one correspondence.

As in the undirected case, for a given $G$,
every connected component is classified as
 a ``cycle'' (a directed path without terminals) and
 a ``chain'' (a directed path with a starting terminal and an ending terminal).
For the map $\pi$, the number of chains is $n-m$,
where $n=|V|$, $m=|V_1|$.
The number of cycles of $G$ is denoted by $\len(G)$.
Note that $\len(G)\le m$.
Let $(j_1,k_1),\ldots,(j_{n-m},k_{n-m})\in V\times V$ be directed pairs of
ending and starting terminal vertices of $n-m$ chains of $G$, and let
$$ \check E = \{(j_1,k_1),\ldots,(j_{n-m},k_{n-m})\}. $$

Using the notations above, we give the general form for the moments as follows.

\begin{theorem}[Moment of the complex noncentral Wishart distribution]
\label{theorem:moment-complex}
Let $(w_{ij})\sim CW(\nu,(\sigma_{ij}),(\delta_{ij}))$, and let
$\overline{i}=i+n$, $i=1,\ldots,n$.  Then,
\begin{equation}
\label{moment-complex}
E[w_{1\overline{1}}\cdots w_{n\overline{n}}]
= \sum_E \nu^{\len(G)} \sigma^E \delta^{\check E},
\end{equation}
where
\begin{align*}
& \sigma^E = \prod_{(i,i')\in E} \sigma_{i\overline{i'}}
 = \sigma_{i_1{\overline{\pi(i_1)}}} \cdots
 \sigma_{i_m{\overline{\pi(i_m)}}}, \\
& \delta^{\check E} = \prod_{(j,j')\in \check E} \delta_{j\overline{j'}}
 = \delta_{j_1 \overline{j'_1}} \cdots \delta_{j_{n-m} \overline{j'_{n-m}}}.
\end{align*}
The summation $\sum_E$ is taken over all possibilities of
$V_1=\{i_1,\ldots,i_m\} \subset \{1,\ldots,n\}$, and injections $\pi:V_1\to V$.
\end{theorem}

\begin{example}
Consider the evaluation of the moment
$E[w_{1\overline{1}}w_{2\overline{2}}w_{3\overline{3}}]$.
Then, $V=\{1,2,3\}$.
There are 34 injections from subsets $V_1\subset V$ to $V$.
Figure \ref{figure:Gpi} is the graph
$G=(V,E)$ for $E=\{(1,1),(2,3)\}$ ($\check E=\{(3,2)\}$).
Summing up 34 possibilities, we have the following:
\begin{align*}
E[w_{1\overline{1}} w_{2\overline{2}} w_{3\overline{3}}] =
& \nu^3 \sigma_{1\overline{1}} \sigma_{2\overline{2}} \sigma_{3\overline{3}}
 +\nu^2 \sigma_{1\overline{2}} \sigma_{2\overline{1}} \sigma_{3\overline{3}} [3]
 +\nu \sigma_{1\overline{2}} \sigma_{2\overline{3}} \sigma_{3\overline{1}} [2] \\
&+\nu^2 \sigma_{1\overline{1}} \sigma_{2\overline{2}} \delta_{3\overline{3}} [3]
 +\nu \sigma_{1\overline{2}} \sigma_{2\overline{1}} \delta_{3\overline{3}} [3]
 +\nu \sigma_{1\overline{1}} \sigma_{2\overline{3}} \delta_{3\overline{2}} [6]
 +\sigma_{1\overline{2}} \sigma_{2\overline{3}} \delta_{3\overline{1}} [6] \\
&+\nu \sigma_{1\overline{1}} \delta_{2\overline{2}} \delta_{3\overline{3}} [3]
 +\sigma_{1\overline{2}} \delta_{2\overline{1}} \delta_{3\overline{3}} [6] \\
&+\delta_{1\overline{1}} \delta_{2\overline{2}} \delta_{3\overline{3}}.
\end{align*}
Here, $[n]$ means that there are $n$ terms of similar form.
\end{example}

\begin{figure*}
\begin{center}
\scalebox{0.4}{\includegraphics{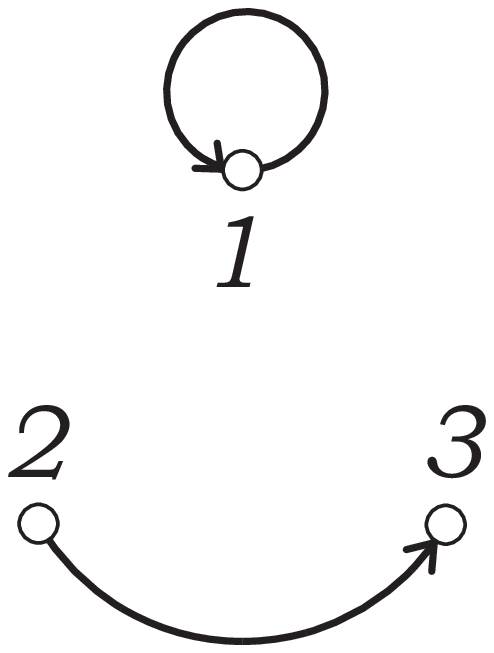}} 
\caption{Directed graph $G=(V,E)$ presenting the term
 $\nu^1 \sigma_{1\overline{1}} \sigma_{2\overline{3}} \delta_{3\overline{2}}$
($n=3$, $m=2$, $\len(G)=1$).}
\label{figure:Gpi}
\end{center}
\end{figure*}

\begin{proof}
Note that $w_{ij}=\sum_{t=1}^\nu z_{ti} \overline{z_{tj}}$.
In view of Lemma \ref{lemma:wick-complex}, we have
\begin{align}
E[w_{1\overline{1}} & \cdots w_{n\overline{n}}] \nonumber \\
=& \sum_{t_1=1}^\nu\cdots\sum_{t_n=1}^\nu
 E[z_{t_1,1} \overline{z_{t_1,\overline{1}}}\cdots z_{t_n,n} \overline{z_{t_n,\overline{n}}}]
 \nonumber \\
=& \sum_E \sum_{t_1}\cdots\sum_{t_n}
 \Cum(z_{t_{i_1},i_1}, \overline{z_{t_{\pi(i_1)},\overline{\pi(i_1)}}}) \cdots
 \Cum(z_{t_{i_m},i_m}, \overline{z_{t_{\pi(i_m)},\overline{\pi(i_m)}}}) \nonumber \\
& \times
 E[z_{t_{i_{m+1}},i_{m+1}}] \cdots E[z_{t_{i_n},i_n}]
 E[\overline{z_{t_{i'_{m+1}},\overline{i'_{m+1}}}}] \cdots E[\overline{z_{t_{i'_n},\overline{i'_n}}}],
\label{sum_E-complex}
\end{align}
where $V_1=\{i_1,\ldots,i_m\}$,
$$ \{i_{m+1},\ldots,i_n\} = V\setminus V_1, \quad
   \{i'_{m+1},\ldots,i'_n\} = V\setminus \pi(V_1). $$
Since $\{i_1,\ldots,i_n\}=V$, the indices $i_1,\ldots,i_n$ can be
divided into connected components of the graph $G$.
A connected component having vertices $j_1,\ldots,j_{k}$
forms either a directed chain
$$ (j_1,j_2),(j_2,j_3),\ldots,(j_{k-2},j_{k-1}),(j_{k-1},j_k) $$
or a directed cycle
$$ (j_1,j_2),(j_2,j_3),\ldots,(j_{k-1},j_k),(j_k,j_1), $$
where $\pi(j_i)=j_{i+1}$ (and $\pi(j_k)=j_{1}$ in the cycle case).
Since the running $n$ indices $t_1,\ldots,t_n$ correspond to
 $n$ vertices of $V$,
the argument of the summation $\sum_E$ in (\ref{sum_E-complex}) is
a product of terms of the form
\begin{align}
\sum_{t_{j_1}} & \sum_{t_{j_2}}\cdots\sum_{t_{j_k}}
 E[\overline{z_{t_{j_1},\overline{j_1}}}]
 \Cum(z_{t_{j_1},j_1}, \overline{z_{t_{j_2},\overline{j_2}}})
 \Cum(z_{t_{j_2},j_2}, \overline{z_{t_{j_3},\overline{j_3}}}) \cdots \nonumber \\
& \times
 \Cum(z_{t_{j_{k-1}},j_{k-1}}, \overline{z_{t_{j_k},j_k}})
 E[z_{t_{j_k},j_k}]
\label{chain-complex}
\end{align}
in the chain case, or
\begin{align}
\sum_{t_{j_1}} & \sum_{t_{j_2}}\cdots\sum_{t_{j_k}}
 \Cum(z_{t_{j_1},j_1}, \overline{z_{t_{j_2},\overline{j_2}}})
 \Cum(z_{t_{j_2},j_2}, \overline{z_{t_{j_3},\overline{j_3}}}) \cdots \nonumber \\
& \times
 \Cum(z_{t_{j_{k-1}},j_{k-1}}, \overline{z_{t_{j_k},\overline{j_k}}})
 \Cum(z_{t_{j_k},j_k}, \overline{z_{t_{j_1},\overline{j_1}}})
\label{cycle-complex}
\end{align}
in the cycle case.
Noting that
$$ \Cum(z_{si},\overline{z_{t\overline{j}}})=\Cov(z_{si},z_{t\overline{j}})=
 1_{\{s=t\}} \sigma_{i\overline{j}} $$
and
$\sum_{t=1}^\nu E[z_{ti}] E[\overline{z_{t\overline{j}}}]=\delta_{i\overline{j}}$,
we see that
$(\ref{chain-complex}) =
 \sigma_{j_1 \overline{j_2}} \sigma_{j_2 \overline{j_3}}
 \cdots \sigma_{j_{k-1} \overline{j_k}} \delta_{j_k \overline{j_1}}$
and
$ (\ref{cycle-complex}) = \nu
 \sigma_{j_1 \overline{j_2}} \sigma_{j_2 \overline{j_3}}
 \cdots \sigma_{j_{k-1} \overline{j_k}} \sigma_{j_k \overline{j_1}}$.
This completes the proof.
\QED
\end{proof}

\subsection{Enumeration of directed graphs}

In this subsection, we evaluate the number of directed graphs appearing in the
expression (\ref{moment-complex})
for the moments of the complex Wishart distribution.

Let $V=\{1,\ldots,n\}$.  Let $V_1$ be a subset of $V$ such that $|V_1|=m$.
Let $\pi$ be an injection $V_1\to V$.
As explained in the previous subsection, we can define a directed graph
$G=(V,E)$ with $E=\{ (i,\pi(i)) \mid i\in V_1 \}$.
The connected components of $G$ are either a directed cycle or
a directed chain (the length may be 0).
Note that the number of chains is $n-m$.
Let $g_{l,m,n}$ be the number of such graphs having $l$ cycles.
The coefficient $g_{l,m,n}$ satisfies the following recurrence formula.

\begin{lemma}
\begin{equation}
\label{recurrence-g}
 g_{l,m,n} = g_{l-1,m-1,n-1} + g_{l,m,n-1} + (2n-m-1) g_{l,m-1,n-1}
\end{equation}
with boundary conditions
\begin{equation}
\label{boundary-g-1}
 g_{l,0,n} = \begin{cases} 1 & (l=0), \\ 0 & (l\ge 1) \end{cases}
 \quad \text{for $n\ge 1$},
\end{equation}
and
\begin{equation}
\label{boundary-g-2}
 g_{l,1,1} = \begin{cases} 0 & (l=0), \\ 1 & (l=1). \end{cases}
\end{equation}
\end{lemma}

\begin{proof}
We consider removing the vertex $n$ and the adjacent edges from the graph $G$.
There are three types of status about adjacent edges.

Case (i).
$\pi(n)=n$.
In this case, the vertex $n$ is contained in a cycle with length 1.
Removing the vertex $n$ and the edge $(n,\pi(n))$ yields
a graph whose values of $l,m$ and $n$ are one less than those of $G$.
This corresponds to the first term in the right-hand side of
 (\ref{recurrence-g}).

Case (ii).
Neither $\pi(n)$ nor $\pi^{-1}(n)$ exists.
In this case, the vertex $n$ is an isolated vertex.
Removing the vertex $n$ yields a graph
whose $l$ and $m$ are invariant and $n$ are one less than that of $G$.
This corresponds to the second term in the right-hand side of
 (\ref{recurrence-g}).

Case (iii).
Otherwise.
In this case,
the vertex $n$ is contained in a cycle with length more than or equal to 2,
or contained in a chain with length more than or equal to 1.
Remove the vertex $n$ and one edge adjacent to the vertex $n$.
This manipulation yields a graph whose $l$ is invariant, and
whose $m$ and $n$ are one less than those of $G$.
Conversely, when we rebuild the graph $G$ from this smaller graph,
there are $(n-1)+\{(n-1)-(m-1)\}=2n-m-1$ places where the vertex $n$ and
one edge can be inserted.
 This corresponds to the third term in the right-hand side of
 (\ref{recurrence-g}).
\QED
\end{proof}

\begin{theorem}
The generating function of the coefficient $g_{l,m,n}$ with respect to
the cycle number $l$,
$$ \Psi_{m,n}(\nu) = \sum_{l\ge 0} \nu^l g_{l,m,n}, $$
is given by
\begin{equation}
\label{solution-psi}
 \Psi_{m,n}(\nu) = {n \choose m} \prod_{i=1}^m (\nu +n-i)
 \quad (0\le m\le n,\,n\ge 1).
\end{equation}
Here, we use a convention $\prod_{i=1}^0=1$.
\end{theorem}

\begin{proof}
Because $g_{-1,m,n}=0$, we see
\begin{equation}
\label{recurrence-psi}
 \Psi_{m,n} = \nu\Psi_{m-1,n-1} + \Psi_{m,n-1} + (2n-m-1) \Psi_{m-1,n-1} .
\end{equation}
The boundary conditions
$\Psi_{0,n}=1$ ($n\ge 1$) due to (\ref{boundary-g-1})
and $\Psi_{1,1}=\nu$ due to (\ref{boundary-g-2}) 
satisfy (\ref{solution-psi}).
In addition, since $\Psi_{n,n-1}=0$,
$$ \Psi_{n,n}=(\nu+n-1)\Psi_{n-1,n-1}=\cdots=\prod_{i=1}^n (\nu +n-i) \quad
 (n\ge 1). $$
Hence, it is sufficient to ensure that (\ref{solution-psi}) satisfies
 (\ref{recurrence-psi}).
Indeed, we have
\begin{align*}
\Psi_{m,n} - \Psi_{m,n-1}
=&
 {n \choose m} \prod_{i=1}^m (\nu +n-i)
-{n-1 \choose m} \prod_{i=1}^m (\nu +n-1-i) \\
=&
 {n-1 \choose m-1}\frac{1}{m} \prod_{i=1}^{m-1} (\nu +n-1-i) \\
& \qquad \times
 \{ n (\nu +n-1)-(n-m) (\nu +n-1-m) \} \\
=& (2n-m-1) \Psi_{m-1,n-1} + \nu\Psi_{m-1,n-1}.
\end{align*}
%
\QED
\end{proof}

\begin{remark}
Comparing (\ref{solution-psi}) with
 (\ref{mgf-stirling}) in Remark \ref{remark:f-stirling}, we have
$$ g_{l,m,n} = {n \choose m} s_n(m,l), $$
where
$s_n(m,l)$ is the noncentral Stirling numbers of the first kind.
In particular, $g_{l,n,n} = s_n(n,l)$
is the Stirling numbers of the first kind.
Let $V=\{1,\ldots,n\}$ and let $\pi : V\to V$ (bijection).
It is well-known that the Stirling number of the first kind $s_n(n,l)$
is the number of directed graphs $(V,E)$,
$E = \{(i,\pi(i)) \mid i \in V \}$ having $l$ cycles (\cite{Stanley2000}).
\end{remark}

\begin{corollary}
$$ \Psi_{m,n}(1)
= {n \choose m} n(n-1)\cdots (n-m+1)
= {n \choose m}^2 m! $$
is the number of directed graphs $G$, and
$$
\Psi_{m,n}(0)
= {n \choose m} (n-1)(n-2)\cdots (n-m) \\
= \frac{n! (n-1)!}{m! \, (n-m)! (n-m-1)!}
$$
is the number of directed graphs $G$ without cycles.
\end{corollary}

\subsection{Degenerate cases}

\subsubsection{The noncentral chi-square distribution}

As in the real case, we can obtain several identities for moments
by assuming that the parameters $\Sigma$ and $\Delta$ have
particular kinds of structures.
First, we consider the case where
$\Sigma=(\sigma_{ij})$, $\sigma_{ij}\equiv 2$,
and 
$\Delta=(\delta_{ij})$, $\delta_{ij}\equiv \delta$.
Then, every element of $W$ has the same value $w$, say, with probability one,
and
the distribution of $w$ is the noncentral chi-square distribution
 $\chi^2_{2\nu}(\delta)$ with $2\nu$ degrees of freedom and the noncentrality
parameter $\delta$.
The $n$th moment of $w\sim\chi^2_{2\nu}(\delta)$ is
\begin{align*}
E[w^n]
& = E[w_{1\overline{1}} \cdots w_{n\overline{n}}] \\
& = \sum_{m=0}^n \sum_{l\ge 0} \nu^l g_{l,m,n} 2^m \delta^{n-m}
  = \sum_{m=0}^n \Psi_{m,n}(\nu) 2^m \delta^{n-m} \\
& = \sum_{m=0}^n {n \choose m} \prod_{i=1}^m(\nu+n-i) 2^m \delta^{n-m}
  = \sum_{m=0}^n {n \choose m} \prod_{i=1}^m(2\nu+2(n-i)) \delta^{n-m}.
\end{align*}
This coincides with the formula (\ref{moment-noncentralchisq})
obtained from the real Wishart distribution.

\subsubsection{Moments of a bivariate chi-square distribution
 associated with the complex Wishart distribution}

Let $(w_{ij})\sim CW_2(\nu,\Sigma)$ be a $2\times 2$ central complex
Wishart matrix.  We consider a particular structure of the parameter:
$$ \Sigma = 2 \begin{pmatrix} 1 & \overline\rho \\ \rho & 1 \end{pmatrix},
 \quad\rho\in\C. $$
The diagonal elements $(w_{11},w_{22})$ are distributed according
to a sort of bivariate chi-square distribution, since the marginal distribution
of $w_{11}$ and $w_{22}$ are the chi-square distribution $\chi^2_{2\nu}$
and they are correlated.  At a glance, $(w_{11},w_{22})$ has a different
distribution from Kibble's distribution since it has a different origin.
However, from (\ref{mgf-complex}), the moment generating function is
\begin{align*}
E[e^{ t_{11}w_{11} +  t_{22}w_{22} }]
&= \det\begin{pmatrix} 1 - 2 t_{11} & - 2 \overline\rho t_{11} \\
                     -2 \rho t_{22} & 1 - 2 t_{22} \end{pmatrix}^{-\nu} \\
&= (1 - 2 t_{11} - 2 t_{22} + 4 t_{11} t_{22} - 4 \rho\overline{\rho}
 t_{11} t_{22})^{-\nu},
\end{align*}
which is equal to the moment generating function (\ref{mgf-bivariategamma})
of Kibble's bivariate chi-square distribution with $\nu$ and $\rho$
replaced by $2\nu$ and $\sqrt{\rho\overline{\rho}}$, respectively.

\subsubsection{Moments of a $2\times 2$ complex Wishart distribution}

Consider a $2\times 2$ complex Wishart matrix
$(w_{ij})\sim CW_2(\nu,I)$.
We first show that
$$ E[w_{12}^a w_{21}^{a'} w_{11}^b w_{22}^c]=0 \quad\text{if $a\ne a'$}. $$
Let $Z_t$ be complex Gaussian random variables from which the Wishart
matrix $W$ is constructed.  Since $E[Z_t]=0$, its distribution is invariant
when the first element of $Z_t$ is multiplied by $\ii$.
On the other hand, this manipulation causes
$$ \begin{pmatrix} w_{11} & w_{12} \\ w_{21} & w_{22} \end{pmatrix} 
 \mapsto 
   \begin{pmatrix} w_{11} & \ii w_{12} \\ -\ii w_{21} & w_{22} \end{pmatrix}.
$$
Therefore,
\begin{align*}
E[w_{12}^a w_{21}^{a'} w_{11}^b w_{22}^c] 
&= E[(\ii w_{12})^a (-\ii w_{21})^{a'} w_{11}^b w_{22}^c] \\
&= \ii^{a-a'} E[w_{12}^a w_{21}^{a'} w_{11}^b w_{22}^c].
\end{align*}
The left- and right-hand sides become 0 unless $a=a'$.

Substituting
$\Theta=(t_{ij})'$ and $\Delta=0$ into (\ref{mgf-complex}),
we obtain the moment generating function
\begin{align}
E[e^{t_{12} w_{12}+t_{21} w_{21}+t_{11} w_{11}+t_{22} w_{22}}]
&= \det\begin{pmatrix} 1 - t_{11} & -t_{12} \\
                          -t_{21} & 1 - t_{22} \end{pmatrix}^{-\nu}
\nonumber \\
&= (1 - t_{11} - t_{22} + t_{11} t_{22} - t_{12} t_{21})^{-\nu}.
\label{mgf-2x2-complex}
\end{align}
Let
$$ \phi_\nu(u,v,w) = (1-u-v+uv-w)^{-\nu}, \quad 
 \phi_\nu^{(k)}(u,v,w) = \left(\frac{\partial}{\partial w}\right)^k
 \phi_\nu(u,v,w). $$
Then,
$$
(\ref{mgf-2x2-complex}) = \phi_\nu(t_{11},t_{22},t_{12}t_{21})
 = \sum_{k\ge 0} \frac{(t_{12}t_{21})^k}{k!} \phi_\nu^{(k)}(t_{11},t_{22},0).
$$
On the other hand, the moment generating function
of the $2\times 2$ real Wishart distribution is rewritten as
$$
(\ref{mgf-2x2-real}) = \phi_{\nu/2}(2t_{11},2t_{22},t_{12}^2)
 = \sum_{k\ge 0} \frac{t_{12}^{2k}}{k!} \phi_{\nu/2}^{(k)}(2t_{11},2t_{22},0).
$$
Comparing the two functions, we can immediately obtain
the moments of the $2\times 2$ complex Wishart distribution
from the results for the real case in Theorem \ref{theorem:2x2wishart}.

\begin{theorem}[Moment of the $2\times 2$ complex Wishart distribution]
Let $a,b$, and $c$ be nonnegative integers.  Then,
\begin{align*}
E[w_{11}^b (w_{12}w_{21})^a w_{22}^c]
= a! \prod_{i=1}^{a}(\nu+a-i)
 \prod_{i=1}^{b}(\nu+a+b-i) \prod_{i=1}^{c}(\nu+a+c-i).
\end{align*}
\end{theorem}


\begin{thebibliography}{99}

%

\bibitem[Bai (1999)]{Bai1999}
Bai, Z.\,D.\ (1999).
Methodologies in spectral analysis of large dimensional random matrices,
 A review.
{\em Statist. Sinica\/}, {\bf 9}, 611--677.





\bibitem[Goodman (1963)]{Goodman1963}
Goodman, N.\,R.\ (1963).
Statistical analysis based on a certain multivariate complex Gaussian
 distribution (An introduction).
{\em Ann.\ Math.\ Statist.\/}, {\bf 34}, 
152--177.

\bibitem[Graczyk, et al.\ (2003)]{Graczyk-etal2003}
Graczyk, P., Letac, G.\ and Massam, H.\ (2003).
The complex Wishart distribution and the symmetric groups.
{\em Ann.\ Statist.\/}, {\bf 31}, 287--309.

\bibitem[Graczyk, et al.\ (2005)]{Graczyk-etal2005}
Graczyk, P., Letac, G.\ and Massam, H.\ (2005).
The hyperoctahedral group, symmetric group representations and the moments
of the real Wishart distribution.
{\em J.\ Theor.\ Probab.\/}, {\bf 18}, 1--42.

\bibitem[Johnson, et al.(1995)]{Johnson-etal1995}
Johnson, N.\,L., Kotz, S.\ and Balakrishnan, N.\ (1995).
{\em Continuous Univariate Distributions\/}, Vol.\ 2, 2nd ed. 
Wiley-Interscience.

\bibitem[Kibble (1941)]{Kibble1941}
Kibble, W.\,F.\ (1941).
A two-variate gamma type distribution.
{\em Sankhya\/}, {\bf 5A}, 137--150.

\bibitem[Koutras (1982)]{Koutras1982}
Koutras, M.\ (1982).
Noncentral Stirling numbers and some applications.
{\em Discrete Math.\/}, {\bf 42}, 73--89.

\bibitem[Kuriki and Takemura (1996)]{Kuriki-Takemura1996}
Kuriki, S.\ and Takemura, A.\ (1996).
Asymptotic expansion of null distribution of likelihood ratio statistic
 in multiparameter exponential family to an arbitrary order.
{\em Probability Theory and Mathematical Statistics:
 Proceedings of the Seventh Japan-Russia Symposium\/}, 244--255.

\bibitem[Letac and Massam (2008)]{Letac-Massam2008}
Letac, G.\ and Massam, H.\ (2008).
The noncentral Wishart as an exponential family, and its moments.
{\em J.\ Multivariate Anal.\/}, {\bf 99}, 1393--1417.

\bibitem[Lu and Richards (2001)]{Lu-Richards2001}
Lu, I-L.\ and Richards, D.\,St.\,P.\ (2001).
MacMahon's master theorem, representation theory, and moments of
 Wishart distributions. 
{\em Adv.\ Appl.\ Math.\/}, {\bf 27}, 531--547.

\bibitem[Maiwald and Kraus (2000)]{Maiwald-Kraus2000}
Maiwald, D.\ and Kraus, D.\ (2000).
Calculation of moments of complex Wishart and complex inverse Wishart
 distributed matrices.
{\em IEE Proc.-Radar, Sonar Navig\/}, {\bf 147}, 162--168.

\bibitem[McCullagh (1987)]{McCullagh1987}
McCullagh, P.\ (1987).
{\em Tensor Methods in Statistics\/}.
Chapman \& Hall/CRC.

\bibitem[Morris (1982)]{Morris82}
Morris, C.\,N.\ (1982).
Natural exponential families with quadratic variance functions.
{\em Ann.\ Statist.\/}, {\bf 10}, 65--80.

\bibitem[Muirhead (1982)]{Muirhead1982}
Muirhead, R.\,J.\ (1982).
{\em Aspects of Multivariate Statistical Theory\/}.
John Wiley \& Sons.

\bibitem[Nadarajah and Kotz(2006)]{Nadarajah-Kotz2006}
Nadarajah, S.\ and Kotz, S.\ (2006).
Product moments of Kibble's bivariate gamma distribution.
{\em Circuits Systems Signal Process.\/}, {\bf 25}, 567--570.


\bibitem[Stanley (2000)]{Stanley2000}
Stanley, R.\,P.\ (2000).
{\em Enumerative Combinatorics\/}, Vol.\ 1, 2nd ed.
Cambridge Univ.\ Press.

\bibitem[Takemura (1991)]{Takemura1991}
Takemura, A.\ (1991).
{\em Foundations of Multivariate Statistical Inference\/} (in Japanese). 
Kyoritsu Shuppan.

\bibitem[Vere-Jones (1988)]{Vere-Jones1988}
Vere-Jones, D.\ (1988).
A generalization of permanents and determinants.
{\em Linear Algebra Appl.\/}, {\bf 111}, 119--124.


\bibitem[Wishart (1928)]{Wishart1928}
Wishart, J.\ (1928).
The generalised product moment distribution in samples from a normal
 multivariate population.
{\em Biometrika\/}, {\bf 20A}, 
32--52.

\bibitem[Withers and Nadarajah (2006)]{Withers-Nadarajah2006}
Withers, C.\ and Nadarajah, S.\ (2006)
Simple representations for Hermite polynomials.
{\em Electronics Letters\/}, {\bf 42}, 1368--1369.

\end{thebibliography}
\end{document}